\documentclass[11pt]{article}
\usepackage{amsfonts}
\usepackage{mathrsfs}
\usepackage{amssymb,amsmath,mathrsfs, amsthm}
\usepackage{palatino,cite}
\usepackage{graphicx}
\usepackage{mathtools}
\usepackage{hyperref,cases,anysize,enumerate}

\numberwithin{equation}{section}

\newtheorem{theorem}{Theorem}[section]
\newtheorem{lemma}{Lemma}[section]
\newtheorem{corollary}{Corollary}[section]
\newtheorem{proposition}{Proposition}[section]
\theoremstyle{definition}
\newtheorem{definition}{Definition}[section]
\newtheorem{remark}{Remark}[section]
\newtheorem{example}{Example}[section]
\newtheorem{question}{Question}[section]

\theoremstyle{remark}

\date{}
\begin{document}

\title{Normal families of holomorphic mappings between complex Finsler manifolds}
\author{Jun Nie (jnie@ncu.edu.cn)\\
Department of Mathematics, Nanchang University\\ Nanchang,  330031, China\\
}

\date{}
\maketitle

\begin{abstract}
In this paper, we find that the integrated form $d_F$ of a complex Finsler metric $F$ is inner. The distance $d_F$ is complete if and only if every closed bounded subset of a complex manifold $M$ is compact. We prove a version of theorem for normal families of holomorphic mappings between two complex Finsler manifolds, i.e, the theorem of Montel in complex Finsler manifolds. Our results extend the basic theorem of strongly negatively curved families for a Hermitian manifold [Wu, Acta Math. 119(1967), 193-233] or [Grauert, Reckziegel, Math. Z. 89(1965), 108-125]. As applications, we obtain a complex Finsler version of theorems $A$-$F$ in [Wu, Acta Math. 119(1967), 193-233], including the Cartan-Carath\'eodory-Kaup-Wu theorem, the theorem of the automorphism group on a complex Finsler manifold and some rigid results.
\end{abstract}



\section{Introduction and main results}
\noindent

Denote by $U$, $\mathcal{F} $ a domain in $\mathbb{C}^n$ and a family of holomorphic functions on $U$ respectively. $\mathcal{F}$ is called $\textit{uniformly bounded}$ if and only if there is a constant $C$ such that $\sup_{z\in U}||f(z)|| \leq C$ for all $f \in \mathcal{F}$. Here $||\cdot||$ denotes the ordinary Euclidean norm of $\mathbb{C}^n$. In one complex variable, the theorem of Montel \cite{KS} states a uniformly bounded family $\mathcal{F}$ of holomorphic function is normal. Subsequently, Theorem of Montel was extended to several complex variables. We now introduce a more generalized version of the theorem of Montel as follows.
\begin{theorem}(see \cite{WHS})
Let $M$ be a complex manifold. A uniformly bounded family of holomorphic mappings $\mathcal{F}$ from $M$ into $\mathbb{C}^n$ is equicontinuous and hence relatively compact in $\mathcal{H}(M,\mathbb{C}^n)$.
\end{theorem}
\begin{remark}
If $M$ is a domain in $\mathbb{C}^n$, by Lemma \ref{normal}, then a uniformly bounded family of holomorphic mappings $\mathcal{F}$ from $\mathcal{D}$ into $\mathbb{C}^n$ is normal \cite{KS}. This is the theorem of Montel in several complex variables.
\end{remark}

The theorem of Montel is crucial for the study of holomorphic functions, as it provides a powerful tool for understanding the behavior of sequences of such functions. It essentially says that if a family of holomorphic functions doesn't grow too large within a given domain, then it must be relatively compact in the space of all holomorphic functions on that domain, with respect to the topology of uniform convergence on compact sets. This theorem is a cornerstone in the proof of many fundamental results in complex analysis.

The notion has long proved its importance in the theory of holomorphic functions of one or several complex variable, but its study in the general setting of complex manifolds was begun in 1965 by Grauert and Reckziegel \cite{GR}. They considered the target domain is a complex manifold with a
complete differential metric (a generalized Hermitian metric, which is discussed in detail in \cite{GR}), such that the Gaussian curvature on every one-dimensional complex analytic surface in the target domain is $ \leq K < 0$ for some constant $K$. Then they obtain the corresponding theorem, i.e., Corollary \ref{tight}. In 1967, Wu obtained the main result in \cite{WHS}, which entails as a corollary the main theorem of Grauert-Reckziegel \cite{GR}, Satz 1, in the following.
\begin{theorem}\label{T-1.1}(see \cite{WHS})
A strongly negatively curved family $\mathcal{F}$ of holomorphic mappings from a complex manifold $M$ into a Hermitian manifold $N$ is equicontinuous. If $N$ is complete, then $\mathcal{F}$ is normal.
\end{theorem}

Finsler geometry is Riemannian geometry without quadratic restrictions \cite{Ch}. However, complex Finsler geometry is Hermitian geometry without Hermitian-quadratic restrictions which contains Hermitian geometry as its special case. It is known that for any complex manifold, there are natural intrinsic pseudo-metrics, i.e., the Kobayashi and Carath$\acute{\mbox{e}}$odory pseudo-metrics.  In general, they are only complex Finsler metrics in nature which are not complex Finsler metrics in the sense of Abate and Patrizio (see \cite{abate}). By Theorem \ref{T-1.1}, Theorems $A$-$F$ are rewritten in Hermitian version in \cite{WHS}. The two main problems contemplated here can be briefly described as follows.
\begin{question}\label{Q-1}
What is the \textit{ basic theorem} of normal families of holomorphic mappings between complex Finsler manifolds?
\end{question}
\begin{question}\label{Q-2}
Can theorems $A$-$F$ be generalized to a complex Finsler manifold?
\end{question}

The Hopf-Rinow theorem and Schwarz lemma play an important role in the proof of the \textit{ basic theorem} i.e., Theorem \ref{T-1.1}.

Complex Finsler geometry is Hermitian geometry without Hermitian-quadratic restrictions. The complex Finsler metric is only a length function and not a real Finsler metric. Therefore, we do not directly apply the Hopf-Rinow theorem from Hermitian geometry or real Finsler geomoetry in complex Finsler geometry. Fortunately, according to \cite{kobayashi1}, we have the following theorem for a complex Finsler metric (or a length function) that holds true. Based on the proof of the\textit{ basic theorem} in \cite{WHS}, Theorem \ref{T-1.2} can complete the proof of Theorem \ref{T-1.3}.
\begin{theorem}\label{T-1.2}(cf. Theorem \ref{hopf})
Suppose that $(M,G=F^2)$ is a complex Finsler manifold. And let $ d_G$ be the integrated form of $F$. Then the following statements are equivalent:\\
$(1)$~~the distance $d_F$ is complete;\\
$(2)$~~closed bounded subsets of $M$ are compact.
\end{theorem}

In \cite{GR,WHS}, they considered the Gaussian curvature on every one-dimensional complex analytic surface in a Hermitian manifold $N$. By Propositions \ref{P-2.1}, \ref{P-2.3}, we only need to consider the holomorphic sectional curvature of the target manifold. For more details, we see Definitions \ref{D-4.7},
\ref{D-4.8}, \ref{D-4.9}, \ref{D-4.11}, \ref{D-4.12}.

From differential geometric viewpoints, the generalizations of Schwarz lemma date back to \cite{ahlfors,chen,tosatti,royden,yau2,Kobayashi67,Koranyi,chenya,La}, we refer to Kim-Lee \cite{KL} and references therein. The objects of these generalizations are Hermitian manifolds. Last decade, Schwarz lemmas are generalized to a complex Finsler metrics which are not necessary Hermitian quadric form. In 2013, Shen and Shen \cite{shen} got a Schwarz lemma from a compact complex Finsler manifold into a complex Finsler manifold. In 2019, Wan \cite{wan} obtained a Schwarz lemma from a complete Riemann surface into a complex Finsler manifold. In 2022, the author and Zhong \cite{NZ1,NZ2} proved a Schwarz lemma of holomorphic mappings from a complete K\"ahler manifold to a complex Finsler manifold and then established a Schwarz lemma on a strongly convex weakly K\"ahler-Finsler manifold, where the assumptions of the radial flag curvatures of the complex Finsler metric are required. Recently, Li, Qiu and Zhang obtain the Schwarz lemma between two complex Finsler manifolds \cite{LQZ}. In general, the Schwarz lemmas on a Hermitan manifold rely heavily on the Chern-Lu formula \cite{LuV}, but the Chern-Lu formula on complex Finsler manifolds fails. The generalization of Schwarz lemmas from a Hermitian manifold (or complex Finsler manifold) to a complex Finsler manifold without a Hermitian non-quadratic form metric is nontrivial. The Schwarz lemmas on a manifold with a quadratic form mostly applies the maximum principle to the source manifold, but it does not seem to apply to a complex Finsler manifold. This is because the complex Finsler metric is defined on the holomorphic tangent bundle or projective tangent bundle, not on the source manifold. For more details, we refer to \cite{Nie}.

Combining Theorem \ref{T-1.2} with Schwarz lemma, i.e., Theorem \ref{SL}, we generalize Theorem \ref{T-1.1} to a complex Finsler manifold and obtain the following theorem which answer the Question \ref{Q-1}.
\begin{theorem}\label{T-1.3}(cf. Theorem \ref{basic})
A strongly negatively curved family $\mathcal{F}$ of holomorphic mappings from a complex manifold $M$ into a complex Finsler manifold $(N,G=F^2)$ is equicontinuous. If $(N,G=F^2)$ is complete, then $\mathcal{F}$ is normal.
\end{theorem}

As a simple application of Theorem \ref{T-1.3}, we obtain the following corollary.
\begin{corollary}\label{C-1.1}(cf. Corollary \ref{tight})
If $N$ is a strongly negatively curved complex Finsler manifold, then it is tight. If $N$ is furthermore complete, then it is taut.
\end{corollary}

There are many important applications of Theorem \ref{T-1.3} and Corollary \ref{C-1.1}. Generalizations of the unit disk Schwarz lemma to several complex variables originated from the work of Cartan \cite{cartan} and Carath\'eodory \cite{Caratheodory}. Look \cite{Lb} first studied the properties of the Jacobian matrix of holomorphic mappings in several complex variables, which extends the result of Cartan. In 1967, it was extended to Hermitian manifold by Wu \cite{WHS}. Since then, various generalizations were made to several complex variables. We refer to Kaup \cite{KP}, Eisenman \cite{ED}. It is well-known Cartan-Carath\'eodory-Kaup-Wu theorem. Note that recently the study of the Schwarz lemma at the boundary of various type of domains in $\mathbb{C}^n$ also attracts lots of interests, we refer to Burns-Krantz \cite{BK}, Liu-Tang \cite{lt2,lt4,lt5} ,Liu-Tang-Zhang \cite{lt3,lt6}, Wang-Liu-Tang \cite{wang}, etc. It is very natural to ask the question: \textbf{What is Cartan-Carath\'eodory-Kaup-Wu theorem on complex Finsler manifolds?} The following theorem answer this question.

\begin{theorem}\label{T-1.6}(cf. Theorem \ref{C})
Let $M$  be a complex manifold which is either an open relatively compact submanifold of a strongly negatively curved complex Finsler manifold $(N,G=F^2)$, or a complete strongly negatively curved complex Finsler manifold itself. Then for a holomorphic mapping $f: M \rightarrow M$ leaves a point $p$ fixed:\\
$(1)$ $|\det df_p|\leq 1$;\\
$(2)$ If $df_p:T_p^{1,0}M \rightarrow T^{1,0}_pM$ is the identity linear map, then $f$ is the identity mapping of $M \rightarrow M$.\\
$(3)$ $|\det df_p|=1$ if and only if $f$ is an automorphism of $M$.
\end{theorem}

Now, we give a specific example as following example.
\begin{example}
Suppose that $\pi: \mathcal{X} \rightarrow S$ is an effectively parametrized holomorphic family of compact canonically polarized complex manifolds over a complex manifold $S$. And let $S_1$ be a relatively compact open submanifold of a complex manifold $S$. Then for a holomorphic mapping $f: S_1 \rightarrow S_1$ leaves a point $p$ fixed:\\
$(1)$ $|\det df_p|\leq 1$;\\
$(2)$ If $df_p:T_p^{1,0}S_1 \rightarrow T^{1,0}_pS_1$ is the identity linear map, then $f$ is the identity mapping of $S_1 \rightarrow S_1$.\\
$(3)$ $|\det df_p|=1$ if and only if $f$ is an automorphism of $S_1$.
 \end{example}

By Theorem \ref{T-1.3} and Theorem \ref{Wu-1}, we have the following theorem.

\begin{theorem}\label{T-1.4}(cf. Theorem \ref{A})
Let $M$ be a complex manifold. And let $(N,G)$ be a strongly negatively curved complex Finsler manifold with the volume element $\Omega$ and $\dim M=\dim N=n$. Let $\mathcal{F}_a: M \rightarrow N$ be a family of holomorphic mappings with the following properties:\\
$(1)$~~At a fixed point $p_0$ of $M$, $|f^*\Omega/\mu| \geq a >0$ for all $f \in \mathcal{F}_a$ and $\mu$ is a fixed nonzero real vector at $p_0$ of degree $2n$.\\
$(2)$~~$p_0$ gets carried by each $f \in \mathcal{F}_a$ into some fixed compact set $K$ in $N$.\\
Then there is a positive constant $\alpha$ such that each $f \in \mathcal{F}_a$ possesses a univalent ball of radius $\alpha$ around $f(p_0)$.
\end{theorem}

 If a complex manifold $M$ is a Hermitian manifold, by Theorem \ref{Wu-2} and \ref{T-1.4}, we have the following theorem.
\begin{theorem}\label{T-1.5}(cf. Theorem \ref{B})
Let $M$ be a Hermitian manifold.  And let $(N,G)$ be a strongly negatively curved complex Finsler manifold with the volume element $\Omega$ and $\dim M=\dim N=n$. Let $\mathcal{F}_a: M \rightarrow N$ be a family of holomorphic mappings with the following properties:\\
$(1)$~~At a fixed point $p_0$ of $M$, $|f^*\Omega/\mu| \geq a >0$ for all $f \in \mathcal{F}_a$ and $\mu$ is a fixed nonzero real vector at $p_0$ of degree $2n$.\\
$(2)$~~$p_0$ gets carried by each $f \in \mathcal{F}_a$ into some fixed compact set $K$ in $N$.\\
Then there is a positive constant $\lambda$ such that each $f \in \mathcal{F}_a$ is biholomorphic on an open ball of radius $\lambda$ around $p_0$.
\end{theorem}

The automorphism group on a complex manifold is an important notion. There are several documents about an automorphism group on a bounded domain or Hermitian manifold. We can refer to \cite{cartan,cartan1,KP,Caratheodory,CK1,kobayashi2} and references therein. Of course, \textbf{We wondered if the automorphism group on a complex Finsler manifold is a Lie groups?} The following theorem partially affirms this question.

\begin{theorem}\label{T-1.7}(cf. Theorem \ref{D})
If a complex Finsler manifold $N$ is strongly negatively curved, then its automorphism $H(N)$ is a (not necessarily connected) Lie group and the isotropy group of $H(N)$ at a point is a compact Lie group. If $N$ is compact, then $H(N)$ is finite. No complex Lie transformation group of positive dimension acts nontrivially on $N$.
\end{theorem}

 We know that there are no nonconstant bounded entire functions. Naturally, we want to know whether there are no constant holomorphic mappings between $\mathbb{C}^n$ and a complex Finsler manifold. Here is a negative answer as follows.
\begin{theorem}\label{T-1.8}(cf. Theorem \ref{E})
Every holomorphic mappings from $\mathbb{C}^n$ into a strongly negatively curved complex Finsler manifold $N$ reduces to a constant.
\end{theorem}

 In Hermitian geometry, Kobayashi \cite{Kobayashi67} proved that $\mathbb{C}^n$ does not admit a Hermitian metric whose holomorphic sectional curvature is negative and bounded away from zero. Complex Finsler geometry is Hermitian geometry without Hermitian quadratic restrictions. There is naturally a question: \textbf{Does $\mathbb{C}^n$  admit a complex Finsler metric whose holomorphic sectional curvature is negative and bounded away from zero?}
The answer is No as follows.
\begin{corollary}\label{C-1.2} (cf. Corollary \ref{C-5.2})
$\mathbb{C}^n$ cannot be equipped with a complex Finsler metric whose holomorphic sectional curvature is bounded above by a negative constant $-K_0<0$.
\end{corollary}

Finally, we generalize Theorem \ref{Wu-6} to a  complex Finsler manifold. Hence, we have the following theorem.
\begin{theorem}\label{T-1.9}(cf. Theorem \ref{F})
If a domain $E$ in $\mathbb{C}^n$ can be given a complete strongly negatively curved complex Finsler metric, then it is a domain of holomorphy.
\end{theorem}

In this way, we have answered Question \ref{Q-2} using Theorems \ref{T-1.6}, \ref{T-1.4}, \ref{T-1.5}, \ref{T-1.7}, \ref{T-1.8},\ref{T-1.9}.
\section{Preliminaries for complex Finsler geometry}
\noindent

In this section, we recall the following basic definitions and facts on a complex Finsler manifold. For more details, we refer to \cite{abate,kobayashi1}.

Let $M$ be a complex manifold of complex dimension $n$. Let $\{z^1,\cdots,z^n\}$ be a set of local complex coordinates, and let $\{\frac{\partial}{\partial z^{\alpha}}\}_{1 \leq \alpha \leq n}$ be the corresponding natural frame of $T^{1,0}M$. So that any non-zero element in $\widetilde{M}=T^{1,0}M \setminus \{\text{zero section}\}$ can be written as
$$v=v^{\alpha}\frac{\partial}{\partial z^{\alpha}} \in \widetilde{M},$$
where we adopt the summation convention of Einstein. In this way, one gets a local coordinate system on the complex manifold $\widetilde{M}$:
$$(z;v)=(z^1,\cdots,z^n;v^1,\cdots,v^n).$$
\begin{definition}(see \cite{abate})
A complex Finsler metric $G:=F^2$ on a complex manifold $M$ is a continuous function $G:T^{1,0}M \rightarrow [0,+ \infty)$ satisfying\\
(i) $G$ is smooth on $\widetilde{M}:=T^{1,0}M\setminus\{\mbox{zero section}\}$;\\
(ii) $G(z;v) \geq 0$ for all $v \in T_z^{1,0}M$ with $z \in M$, and $G(z;v)=0$ if and only if $v=0$;\\
(iii) $G(z;\zeta v)=|\zeta|^2G(z;v)$ for all $v \in T^{1,0}_zM$ and $\zeta \in \mathbb{C}$.
\end{definition}

\begin{definition}(see \cite{abate})
A complex Finsler metric $G$ is called strongly pseudoconvex if the Levi matrix
\begin{equation*}
(G_{\alpha \overline{\beta}})= \Big(\frac{\partial^2 G}{\partial v^{\alpha}\partial \overline{v}^{\beta}}\Big)
\end{equation*}
is positive definte on $\widetilde{M}$.
\end{definition}
A manifold $M$ endowed with a complex strongly pseudoconvex Finsler metric $G$ is called a complex Finsler manifold in this paper.
Any $C^\infty$ Hermitian metric on a complex manifold $M$ is naturally a strongly pseudoconvex complex Finsler metric. Conversely,
if a complex Finsler metric $G$ on a complex manifold $M$ is $C^\infty$ over the whole holomorphic tangent bundle $T^{1,0}M$, then it is necessary a $C^\infty$ Hermitian metric. That is, for any $(z;v)\in T^{1,0}M$
$$
G(z;v)=g_{\alpha\overline{\beta}}(z)v^\alpha\overline{v}^\beta
$$
for a $C^\infty$ Hermitian tensor $g_{\alpha\overline{\beta}}$ on $M$.
For this reason, in general the non-trivial (non-Hermitian quadratic) examples (see \cite{ZCP}) of complex Finsler metrics are only required to be smooth over the slit holomorphic tangent bundle $\widetilde{M}$.

Let $(M,G)$ be a complex Finsler manifold $M$, and take $v \in \widetilde{M}$. Then the holomorphic sectional curvature $K_G(v)$ of $G$ along a non-zero tangent vector $v$ is given by
\begin{equation*}
K_G(v)=\frac{2}{G(v)^2}\langle\Omega(\chi,\bar{\chi})\chi,\chi\rangle_v.
\end{equation*}
where $\chi=v^\alpha\delta_\alpha$ is the complex radial horizontal vector field and $\Omega$ is the curvature tensor of the Chern-Finsler connection associated to $(M,G)$.

Abate and Patrizio found a phenomenon that the holomorphic sectional curvature sectional curvature of a strongly pseudoconvex complex Finsler metric $G$ is the supremum of the Gaussian curvature of the induced metric through a family of holomorphic mappings.

\begin{proposition}\label{P-2.1} (see \cite{abate,wong}) 
Suppose that $G:T^{1,0}M \rightarrow [0,+\infty)$ a strongly pseudoconvex complex Finsler metric on a complex manifold $M$, take $z \in M$ and $0 \neq v \in T_z^{1,0}M$. Then
$$K_G(v)=\sup\{K(\varphi^*G)(0)\},$$
where the supremum is taken with respect to the family of all holomorphic mappings $\varphi: \Delta \rightarrow M$ with $\varphi(0)=z$ and $\varphi'(0)=\lambda v$ for some $\lambda \in \mathbb{C}^*$; $K(\varphi^*G)(0)$ is the Gaussian curvature of $(\Delta,\varphi^*G)$ at the point $0$.
\end{proposition}

Kobayashi proved the proposition about the holomorphic sectional curvature of a complex Finsler metric as follows.
\begin{proposition} (see\cite{kobayashi})
Let $E$ be a holomorphic vector bundle with a strongly pseudoconvex complex Finsler structure $F$. Let $E'$ be a subbundle of $E$ and $G'$ the restriction of $G$ to $E'$. Then $G'$ is also convex, and its curvature does not exceed the curvature of $F$.
\end{proposition}
 If $E$ is the tangent bundle $T^{1,0}M$ and $G$ is a strongly pseudoconvex complex Finsler metric on a complex manifold $M$. We have the following proposition.
 \begin{proposition}\label{P-2.3}
 Let $M'$ be a complex submanifold of a complex manifold $M$ with a strongly pseudoconvex complex Finsler metric $G$, and let $v$ be a complex line tangent to $M'$. Then the holomorphic sectional curvature of $H_{G'}(v)$ in $M'$ (with respect to the induced metric $G'$) is not greater than the holomorphic sectional curvature $H_G(v)$ of $v$ in $M$.
 \end{proposition}
\begin{remark}\label{R-2.1}
If a complex Finsler manifold $(M,G)$ is from a Hermitian manifold, Proposition \ref{P-2.3} reduces to Lemma 1 in \cite{wuh}.
\end{remark}
\section{Distance and metric on a complex manifold}
\noindent

In this section, we recall  the definition of distance along with related propositions and corollaries. Through the relationships among these definitions and propositions, we establish the Hopf-Rinow theorem for a complex Finsler manifold, namely, Theorem \ref{hopf}. For more details, we refer to \cite{CK,kobayashi1,kobayashi2}.

\begin{definition}(see \cite{kobayashi1})
Let $M$ be a set. A \textit{distance} $d$ on $M$ is a function on $M \times M$ with values in the non-negative real numbers satisfying the following axioms:\\
$(1)$ \quad $d(p,q)=0$ \quad \quad $\text{if only and if}$  $p=q$;\\
$(2)$ \quad $d(p,q)=d(q,p)$, \quad \quad \quad (symmetry axiom);\\
$(3)$ \quad $d(p,r) \leq d(p,q)+ d(q,r)$, \quad (trangular inequality).

A distance $d$ on $M$ induces a topology on $M$ in a natural manner. This topology is Hausdorff if and only if $d$ is a distance.
\end{definition}
We consider only arcwise connected Hausdorff topological spaces and assume the following axiom.

Let $M$ be a topological space with a distance $d$. Given a curve $\gamma(t),$ $a \leq t \leq b$, in $M$, the \textit{length} $L(\gamma)$ of $\gamma$ is defined by
\begin{equation}
L(\gamma)=\sup \sum_{i=1}^{k}d(\gamma(t_{i-1}),\gamma(t_i)),
\end{equation}
where the supremum is taken over all partitions $a=t_0 < t_1 < \cdots < t_k=b$ of the interval $[a,b]$. A curve $\gamma$ is said to be \textit{rectifiable} if its length $L(\gamma)$ is finite. The space $(M,d)$ is said to be \textit{finitely arcwise connected} if every pair of points $p, q$ of $M$ can be joined by a rectifiable curve.
Now we define a new distance $d^i$, called the \textit{inner distance} induced by $d$, by setting
\begin{equation}
d^i(p,q)= \inf L(\gamma),
\end{equation}
where the infimum is take over all $d$-rectifiable curves $\gamma$ joining $p$ and $q$. When $M$ is a complex space, we modify the definition above by taking the infimum over all $d$-rectifiable, \textit{piecewise differentiable $(C^1)$} curves $\gamma$ joining $p$ and $q$. 
From the definition of $d^i$ it follows immediately that $d \leq d^i$.

From the definition of $d^i$ it follows immediately that
\begin{equation}
d(p,q) \leq d^i(p,q) \quad \text{for}\quad  p,q \in M.
\end{equation}
\begin{proposition}(see \cite{kobayashi1})
Let $(M,d)$ be finitely arcwise connected. Then
\begin{equation*}
L^i(\gamma)=L(\gamma) \quad \text{for all rectifiable curves} \quad \gamma,
\end{equation*}
where $L^i$ is the length defined by the induced inner distance $d^i$.
\end{proposition}
\begin{proof}
For the convenience of the reader, the proof of the proposition is restated here.

Consider a partition $a=t_0<t_1<\cdots<t_k=b$ for the curve $\gamma(t)$, $a \leq t \leq b$. Let $\gamma_j(t)$ be the portion of $\gamma$ corresponding to the interval $t_{j-1} \leq t \leq t_{j}$. Then
$$L(\gamma)=\sum_{j=1}^k L(\gamma_j).$$
From the definition of $d^i$, we obtain
$$d^i(\gamma(t_{j-1},\gamma_i)) \leq L(\gamma_j).$$
Hence, we have
$$\sum_{j=1}^k d^i(\gamma(t_{j-1},\gamma(t_j))).$$
Taking the supremum of the left hand side over side all partitions of the interval $[a,b]$, we obtain
$$L^i(\gamma) \leq L(\gamma).$$
From the inequality $d \leq d^i$, we have
$$L(\gamma) \leq L^i(\gamma).$$
\end{proof}
\begin{corollary}\label{C-3.1}(see \cite{kobayashi1})
Let $(M,d)$ be finitely arcwise connected. Then
$$(d^i)^i=d^i.$$
\end{corollary}

\begin{definition}(see \cite{kobayashi1})
We say that a distance $d$ \textit{inner} if $d^i=d$.
\end{definition}

Let $(M,G=F^2)$ be a complex Finsler manifold, we define the length of an admissible curve by
$$L_F(\gamma)=\int_0^1 F(\gamma(t);\gamma'(t))dt$$
and a distance
$$d_F(x,y)=\inf\{L_F(\gamma)\}$$
where $\gamma$ range over all admissible curves joining $x$ and $y$, $\forall x,y \in M$. We shall refer to $d_F$ as the integrated form of a complex Finsler manifold $(M,G=F^2)$. This is also the distance induced by the complex Finsler metric $G=F^2$ on a complex manifold $M$.

\begin{lemma}\label{inner}
Let $(M,G=F^2)$ be a complex Finsler manifold. The integrated form $d_F$ of a complex Finsler manifold $(M,G=F^2)$ is inner.
\end{lemma}
\begin{proof}
By Corollary \ref{C-3.1} and definition of the integrated form $d_F$, we also know that it is inner.
\end{proof}
\begin{remark}
The above lemma can be found in \cite{CK}, but they did not provide the corresponding proof. From \cite{CK}, we also know thta the inner distance $(c_M)^i$ associated to the Carath\'eodory distance $c_M$ coincides with the integrated form $d_{F^C_M}$ of the infinitesimal Carath\'eodory metirc $F^C_M$. On the other hand, the Carath\'eodory distance $c_M$ is in general not inner, i.e., we have in general $c_M <(c_M)^i=d_{F^C_M}$. On the other hand, the Kobayashi distance is inner. For more details, we refer to \cite[Lemma 3.8, Theorem 3.10]{CK}.
\end{remark}

To prove the Hopf-Rinow theorem in the case of a complex Finsler manifold, i.e. Theorem \ref{hopf}, we need to introduce the following concepts and corresponding propositions.
\begin{definition}(see \cite{kobayashi1})
Let $d$ be a distance function on $M$. We say that $(M,d)$ is \textit{Cauchy complete} or simply \textit{complte} if every Cauchy sequence (with respect to $d$) converges. If every closed ball $B(o;r)=\{ p \in M; d(o,p) \leq r\}$ with $o \in M$ and $r>0$ is compact, then $(M,d)$ is said to be \textit{strongly complete} or \textit{finitely compact}. We say that $(M,d)$ is \textit{weakly complete} if for every point $o \in M$ there is an $r>0$ (which depends on $o$) such that $B(o;r)$ is compact.
\end{definition}
\begin{proposition}\label{P-inner}(see \cite{kobayashi1})
Let $d$ be a distance on a locally compact space $M$.\\
$(1)$ Then we have the following implications:
$$\text{strongly complete} \Rightarrow \text{complete}\Rightarrow \text{weakly complete}.$$
$(2)$ If $d$ is inner, then completeness implies strong completeness.
\end{proposition}
\begin{proof}
For the convenience of the reader, the proof of the proposition is restated here.

$(1)$~~ $(i)$~~The strong completeness obviously implies completeness.

~~~~~~~$(ii)$~~complete $\Longrightarrow$ weakly complete.

Assume that $(X, d)$ is not weakly complete. Then there is a point $o \in X$ such that for every $r>0$ the ball $B(o;r)$ is not compact. For each natural number $n$, take a sequence of points $\{p_{nj}\}_{j=1}^\infty$ in $B(o;\frac{1}{n})$ without accumulation points.

We note that all sequences of the type $\{q_n=p_{nj_n}\}_{n=1}^{\infty}$, where $j_n$ are arbitrary natural numbers, are Cauchy and converges to $o$.

Fix a compact neighborhood $U$ of $o$. For each fixed $j$, let $N_j$ be the smallest integer such that $p_{n_j} \in U$ for all $n >N_j$ (so $p_{{N_j}j} \notin U$). We put $A=\sup_jN_j \leq \infty$.

If $A=\infty$, then there is a subsequence $\{j(\lambda)\}_{\lambda=1}^{\infty}$  of  $ \{j\}_{j=1}^{\infty}$ such that $N_{j(\lambda)} \nearrow \infty$ . Then $p_{N_{j(\lambda)}j(\lambda)} \notin U$ contradicts the fact that the sequence $\{p_{n_{j(\lambda)j(\lambda)}}\}_{\lambda=1}^{\infty}$ converge to $o$.

If $A < \infty$, take $n >A$. Then $\{p_{nj}\}_{j=1}^{\infty}$ are in a compact set $U \cap B(o;\frac{1}{n})$ and must have an accumulation point in $B(o;\frac{1}{n})$, which is also a contradiction.

$(2)$~~complete $\Longrightarrow$ strongly complete when $d$ is inner.

By the condition $d$ is an inner and Proposition 1.1.8 in \cite{kobayashi1}, we know that it induces the given topology of $M$.
\end{proof}

By Proposition \ref{P-inner} and Lemma \ref{inner}, we have the following theorem.
\begin{theorem}\label{hopf}
Suppose that $(M,G=F^2)$ is a complex Finsler manifold. And let $ d_F$ be the integrated form of $F$. Then the following statements are equivalent:\\
$(1)$~~the distance $d_F$ is complete;\\
$(2)$~~closed bounded subsets of $M$ are compact.
\end{theorem}

\section{some notions, lemmas, theorems}
\noindent

In this section, we exhibit some definitions and collect several lemmas and theorems, which will be used in the subsequent section. Firstly, we introduce some basic concepts and results in function theory. For more details, we refer to \cite{WHS}.

Let $M, N$ be complex manifolds. Then by definition: $(i)$~$\mathcal{C}(M,N)$ is the set of all \textit{continuous} mappings from $M$ to $N$. $(ii)$~$\mathcal{H}(M,N)$ is the set of all \textit{holomorphic} mappings from $M$ to $N$. For convenience, $\mathcal{C}(M)\equiv \mathcal{C}(M, M)$, and $\mathcal{H}(M)\equiv\mathcal{H}(M,M)$.

\begin{definition}\label{D-4.1}(see \cite{WHS})
A sequence $\{f_i\} \subseteq \mathcal{C}(M,N)$ is called \textit{compactly divergent} if and only if for given any compact $K$ in $M$, and compact $K'$ in $N$, there exists an $i_0$ such that $f_i(K) \cap K' \neq \varnothing$ for all $i \geq i_0$.
\end{definition}
\begin{definition}(see \cite{WHS})
A set is \textit{relatively compact} in $M$ if and only if it has compact closure in $M$.
\end{definition}
\begin{definition}(see \cite{WHS})
A subset $\mathcal{F}$ of $\mathcal{C}(M,N)$ is called \textit{normal} if and only if every sequence of $\mathcal{F}$ contains a subsequence which is either relatively compact in $\mathcal{C}(M,N)$ or compactly divergent.
\end{definition}
\begin{definition}(see \cite{WHS})
Let $d_N$ be the distance function of $N$. Then $\mathcal{F} \subseteq \mathcal{C}(M,N)$ is called an \textit{equcontinuous family} if and only if given any $\varepsilon>0$ and any $p \in M$, there exists a neighborhood $U$ of $p$ such that $q \in U$ implies $d_N(f(p),f(q)) <\varepsilon$ for all $f \in \mathcal{F}$.\end{definition}

\begin{definition}(see \cite{WHS})
A complex manifold $N$ is called \textit{taut} if and only if for every complex manifold $M$, the set of all holomorphic mappings $\mathcal{H}(M,N)$ is normal.
\end{definition}
\begin{definition}(see \cite{WHS})
Let $N$ be a complex manifold and $d$ be a metric on $N$ inducing its topology. Then $(N,d)$ (or if no confusion is possible, just $N$) is \textit{tight} if and only if for every complex manifold $M$, the set of all holomorphic mappings $\mathcal{H}(M,N)$ is equicontinuous.
\end{definition}

\begin{lemma}(see \cite{WHS})\label{normal}
Let $\mathcal{F} \subseteq \mathcal{C}(M,N)$, where $M, N$ are connected, locally compact metric spaces. Then\\
 $(i)$~~If $\mathcal{F}$ is compact, then $\mathcal{F} $ is normal.\\
 $(ii)$~~If $\mathcal{F}$ is normal, then its closure is locally compact.\\
 $(iii)$~~If $\mathcal{F}$ is equicontinuous and if each bounded subset of $N$ is relatively compact, then $\mathcal{F}$ is normal.
 \end{lemma}

\begin{lemma}\label{Ascoli}(see \cite{WHS} Ascoli theorem). 
$\mathcal{F} \subseteq \mathcal{C}(M,N)$ is compact if and only if:\\
$(a)$~~$\mathcal{F}$ is closed in $\mathcal{C}(M,N).$\\
$(b)$~~$\mathcal{F}(p)(\equiv\{n \in N: n=f(p)$ for some $f \in \mathcal{F}\})$ is relatively compact in $N$ for every $p \in M$.\\
$(c)$~~$\mathcal{F}$ is equicontinuous.
\end{lemma}

\begin{lemma}\label{equi}(see \cite{WHS})
$N$ is taut if $\mathcal{H}(B^n, N)$ is normal for all $n$, where $B^n$ is the unit ball in $\mathbb{C}^n$. $(N,d)$ is tight if $\mathcal{H}(B^n, N)$ is equicontinuous for all $n$.
\end{lemma}

\begin{definition}(see \cite{WHS})
If $f:M \rightarrow N$ is a holomorphic and $N$ is Hermitian, then a univalent ball (for $f$) is an open ball in the image of $f$ onto which $f$ maps an open set biholomorphically.
\end{definition}

Now, we introduce the following definitions of \textit{strongly negatively curved family} and \textit{strongly negatively curved} in Hermitian manifold.
\begin{definition}\label{D-4.7}(see \cite{WHS})
Let $\mathcal{F}: M \rightarrow N$ be a family of holomorphic mappings from a complex manifold $M$ into a Hermitian manifold $N$ with Hermitian metric $h$. $\mathcal{F}$ is called a \textit{strongly negatively curved family (of order $-k_0<0$)} if and only if for any $f \in \mathcal{F}$ and for holomorphically imbedded disc $\mathbb{D}$ in $M$ the curvature of the pseudo-Hermitian metric $(f|_{\mathbb{D}})^*h$ is bounded above by $-k_0<0$.
\end{definition}

 By Remark \ref{R-2.1}, Propositions \ref{P-2.1} and \ref{P-2.3}, we can simplify Definition \ref{D-4.7} as follows.
\begin{definition}\label{D-4.8}
Let $\mathcal{F}: M \rightarrow N$ be a family of holomorphic mappings from a complex manifold $M$ into a Hermitian manifold $N$ with Hermitian metric $h$. $\mathcal{F}$ is called a \textit{strongly negatively curved family (of order $-k_0<0$)} if and only if the holomorphic sectional curvature of Hermitian manifold $(N,h)$ is bounded above by $-k_0<0$.
\end{definition}
\begin{definition}(see \cite{WHS})\label{D-4.9}
A Hermitian manifold $(N,h)$is \textit{strongly negatively curved} (of order $-k_0<0$) if and only if the holomorphic sectional curvature of the Hermitian manifold $(N,h)$ is bounded above by $-k_0$.
\end{definition}
By Propositions \ref{P-2.1} and \ref{P-2.3}, we are able to introduce the complex Finsler version of the definitions of \textit{strongly negatively curved family} and \textit{strongly negatively curved}.

\begin{definition}\label{D-4.11}
Let $\mathcal{F}: M \rightarrow N$ be a family of holomorphic mappings from a complex manifold $M$ into a complex Finsler manifold $(N,G=F^2)$. $\mathcal{F}$ is called a \textit{strongly negatively curved family (of order $-k_0<0$)} if and only if the holomorphic sectional curvature of the complex Finsler manifold $(N,G=F^2)$ is bounded above by $-k_0<0$.
\end{definition}

\begin{definition}\label{D-4.12}
A complex Finsler manifold $(N,G=F^2)$is \textit{strongly negatively curved} (of order $-k_0<0$) if and only if the holomorphic sectional curvature of the complex Finsler manifold $(N,G=F^2)$ is bounded above by $-k_0$.
\end{definition}

\begin{theorem}\label{Wu-1} (see \cite{WHS})
If $\mathcal{F}_a$ is a relatively compact family, then there exists a positive constant $\alpha$ such that every $f \in \mathcal{F}_a$ possesses a univalent ball of radius $\alpha$ around $f(p_0)$.
\end{theorem}
\begin{theorem}\label{Wu-2}(see \cite{WHS})
Suppose $M$ also Hermitian. Hypothesis as above, then there is a positive constant $\alpha$ such that every $f \in  \mathcal{F}_a$ is biholomorphic on the open ball of radius $\lambda$ around $p_0$.
\end{theorem}

\begin{theorem}\label{Wu-4}(see \cite{WHS})
If $(N,d)$ is a tight complex manifold, then its automorphism group $\mathcal{H}(N)$ is a (not necessarily connected) Lie group, and the isotropy subgroup of $\mathcal{H}(N)$ at a point is a compact Lie group.
\end{theorem}

\begin{theorem}\label{Wu-5}(see \cite{WHS})
If $(N,d)$ is a tight complex manifold, then there is no nonconstant holomorphic map of $\mathbb{C}^n$ into $N$.
\end{theorem}
\begin{theorem}\label{Wu-6}(see \cite{WHS})
If a domain $E$ (=open connected set) in $\mathbb{C}^n$ is taut and $E\neq \mathbb{C}^n$, then $E$ is pseudoconvex and hence a domain of holomorphy.
\end{theorem}
\begin{theorem}\label{Wu-7} (see \cite{kobayashi})
If $M$ is a complex Finsler manifold $G$ whose holomorphic sectional curvature is bounded from above by a negative constant, then $M$ is hyperbolic.
\end{theorem}
\begin{theorem} \label{Wu-8}(see \cite{kobayashi2})
A connected complex Lie group of holomorphic transformations acting effectively on a hyperbolic manifold $M$ reduces to the identity element only.
\end{theorem}
\section{Main theorems}
\noindent

In this section, by the Schwarz Lemma from a complete K\"ahler manifold into a complex Finsler manifold and Theorem \ref{hopf}, we prove the
complex Finsler version of \textit{basic theorem}, i.e., Theorem \ref{basic}. In order to prove the complex Finsler version of basic theorem, we introduce the Schwarz lemma from a complete K\"ahler manifold into a complex Finsler manifold.
\begin{theorem}\label{SL}(see \cite{NZ1,NZ2,LQZ})
 Suppose that $(M,h)$ is a complete K\"ahler manifold with holomorphic sectional curvature bounded from below by a constant $K_1$ and sectional curvature bounded from below, while $(N,G)$ is a complex Finsler manifold with holomorphic sectional curvature of the Chern-Finsler connection bounded from above by a constant $K_2<0$. Then any holomorphic map $f$ from $M$ into $N$ satisfies
 \begin{equation}
(f^*G)(z;v) \leq \frac{K_1}{K_2}h(z;v), ~~\forall(z;v) \in T^{1,0}M.
\end{equation}
In particular, $K_1 \geq 0,$  then any holomorphic map $f$ from $M$ into $N$ is constant.
\end{theorem}
 Now we introduce and prove the complex Finsler version of basic theorem.
\begin{theorem}\label{basic}
A strongly negatively curved family of $\mathcal{F}$ of holomorphic mappings from a complex manifold $M$ into a complex Finsler manifold $(N,G=F^2)$ is equicontinuous. If $(N,G=F^2)$ is complete, then $\mathcal{F}$ is normal.
\end{theorem}
\begin{proof}
The idea of the following proof was originally given by Wu \cite{WHS} in establishing basic theorem of normal families of holomorphic mappings in Hermitian manifold.

Suppose we know that $(N,G=F^2)$ is complete, then by Theorem \ref{hopf}, each of bounded subset of $N$ is compact and relatively compact. Combining with Lemma \ref{normal}, we know that $\mathcal{F}$ is normal. Hence it suffices to prove the equicontinuity of $\mathcal{F}$, assuming that it is strongly negatively curved. This is an entirely local question, so we fix an arbitrary point $p \in M$, and take $p$-centered coordinate functions $\{z_1,\cdots,z_m\}$ so that the unit ball $B^n(o)=\{\sum_{i=1}^m |z_i|^2 <1\}$ is well defined. For convenience, we assume $\mathcal{F}$ is strongly negatively curved of order $-1$, and that the Bergman metric $h_b$ on the unit ball $B^n(o)$ (which is a K\"ahler metric of constant \textit{holomorphic sectional curvature} $-1$, and its sectional curvature of $(B^n(o), h_b)$ takes values in the interval $[-1, -\frac{1}{4}]$, see Zheng \cite{ZFY}). Applying Theorem \ref{SL}, we have
\begin{equation}
f^*G(z;v) \leq h_b(z;v), ~~\text{for}~~f \in \mathcal{F}, ~~\forall (z;v) \in T^{1,0}M.
\end{equation}
Therefore, for any $q \in B^n$ and $f\in \mathcal{F}$, we have
\begin{equation}
d_F(f(p),f(q)) \leq d_{h_b}(p,q),
\end{equation}
where $d_G$ and $d_{h_b}$ are respectively the integrated form of $G$ and $h_b$.\\
Hence equicontinuous is now obvious: if $\varepsilon$ is given, a ball $B_1$ of radius $\frac{1}{2}\varepsilon$ relative to $d_{h_b}$ around the origin of $B^n$ so that if $q_1, q_2$ are in $B_1$, and $f \in \mathcal{F}$, we have
$$d_F(f(q_1),f(q_2)) \leq d_F(f(q_1),f(o))+d_F(f(q_2),f(o)) \leq d_{h_b}(q_1,o)+d_{h_b}(q_2,o) \leq \varepsilon.$$
Hence, we know $\mathcal{F}$ is equicontinuous.
\end{proof}
As a simple application of Theorem \ref{basic} and Lemma \ref{equi}, we obtain the following result.
\begin{corollary}\label{tight}
If $N$ is a strongly negatively curved complex Finsler manifold, then it is tight. If $N$ is furthermore complete, then it is taut.
\end{corollary}

In 2022, the author and Zhong \cite{NZ1} proved the following theorem.
\begin{theorem}\label{example} (see \cite{NZ1})
Suppose that $D\subset \mathbb{C}^n(n\geq 2)$ is a bounded domain. Then $D$ admits a non-Hermitian quadratic strongly pseudoconvex complex Finsler metric $G:T^{1,0}D\rightarrow \mathbb{R}^+$ such that its holomorphic sectional curvature is bounded from above by a negative constant.
\end{theorem}
By Corollary \ref{tight} and Theorem \ref{example}, we have the following theorem

\begin{theorem}(see \cite{WHS} Theorem of Montel) A uniformly bounded family of holomorphic mappings from a complex manifold $M$ into $\mathbb{C}^n$ is equicontinuous and hence relatively compact in $\mathcal{H}(M, \mathbb{C}^n)$.
\end{theorem}
\begin{proof}
We can assume that $\mathcal{F} \subset \mathcal{H}(M, \mathbb{C}^n)$ is uniformly bounded. Hence, we have
$$\mathcal{F}(M) \subset D_1, $$
where $D_1$ is a bounded domain in $\mathbb{C}^n$.
By Theorem \ref{example}, $D_1$ admits a complex Finsler metric $G$ such that ist holomorphic sectional curvature is bounded from above by a negative constant. Then $(D_1, G)$ is a strongly negatively curved complex Finsler manifold. Combing with Corollary \ref{tight}, we know that $D_1$ is tight. Thus a uniformly bounded family of holomorphic mappings from a complex manifold $M$ into $\mathbb{C}^n$ is equicontinuous and hence relatively compactly in $\mathcal{H}(M,\mathbb{C}^n)$.
\end{proof}

Recently, the author \cite{Nie1} proved the following theorem.
\begin{theorem}\label{T-4.5} (see \cite{Nie1})
Suppose that $\mathcal{D}$ is a bounded pseudoconvex domain with $C^2$-boundary or bounded convex domain in $\mathbb{C}^n$. Then $\mathcal{D}$ admits a complete strongly pseudoconvex complex Finsler metric $G:T^{1,0}\mathcal{D}\rightarrow [0,+\infty)$ such that its holomorphic sectional curvature is bounded from above by a negative constant.
\end{theorem}
By Corollary \ref{tight} and Theorem \ref{T-4.5}, we have the following example.
\begin{example}
Suppose that $\mathcal{D}$ is a bounded pseudoconvex domain with $C^2$-boundary or bounded convex domain in $\mathbb{C}^n$. Then the family of holomorphic mappings from a complex manifold $M$ into $\mathcal{D}$ is normal. What's more, $\mathcal{D}$ is taut.
\end{example}

The above results about a bounded domain in $\mathbb{C}^n$. Now, we provide a more general example. To and Yeung \cite{TY}  proved the following theorem.
\begin{theorem} \label{T-4.7}
Let $\pi: \mathcal{X} \rightarrow S$ be an effectively parametrized holomorphic family of compact canonically polarized complex manifolds over a complex manifold $S$. Then $S$ admits a $C^{\infty}$ Aut $(\pi)$-invariant Finsler metric whose holomorphic sectional curvature is bounded above by a negative constant. As a consequence, $S$ is Kobayashi hyperbolic.
\end{theorem}
By Theorem \ref{T-4.7} and Corollary \ref{tight}, we can obtain the following example.
\begin{example}
Let $\pi: \mathcal{X} \rightarrow S$ be an effectively parametrized holomorphic family of compact canonically polarized complex manifolds over a complex manifold $S$. The family of holomorphic mappings from a complex manifold into a complex manifold $S$ is equicontinuous. Hence base complex manifold $S$ is tight.
\end{example}
\begin{remark}
We also obtian that $S$ admits a complete $C^{\infty}$ Aut $(\pi)$-invariant Finsler metric whose holomorphic sectional curvature is bounded above by a negative constant. what's more, $S$ is a complete Kobayashi-hyperbolic. This will appear elsewhere.
\end{remark}
\section{Cartan-Carath\'eodory-Kaup-Wu theorem on a complex Finsler manifold}
\noindent

In this section, we introduce and prove Theorem \ref{C}. To better elucidate Theorem \ref{C}, we merge Theorems C and C' in \cite{WHS} into a theorem as follows.
\begin{theorem}\label{Wu-3}(see \cite{WHS})
Let $M$ be a relatively compact open submanifold of a complex manifold $N$. Let $d$ be some metric on $N$ such that $(M,d)$ is tight. Or let $M$ be a taut manifold. If $f: M\rightarrow M$ is a holomorphic mapping and $f(p)=p$, consider $df_p: T^{1,0}_pM  \rightarrow T^{1,0}_pM$. Then:\\
$(1)$ $|\det df_p|\leq 1$;\\
$(2)$ If $df_p:T_p^{1,0}M \rightarrow T^{1,0}_pM$ is the identity linear map, then $f$ is the identity mapping of $M \rightarrow M$.\\
$(3)$ $|\det df_p|=1$ if and only if $f$ is an automorphism.\\
\end{theorem}
\begin{proof}
If $M$ is a a relatively compact open submanifold of a complex manifold $N$. Let $d$ be some metric on $N$ such that $(M,d)$ is tight. A detailed proof process is provided by Wu. Now, we prove the case where $M$ is a taut manifold.

Since $M$ is taut, we know that the set of all holomorphic mappings $\mathcal{H}(M)$ is normal. By the definition of normal, we know that every sequence of a subset $\mathcal{F}$ of $\mathcal{H}(M)$ is either relatively compact in $\mathcal{H}(M)$ or compactly divergent. Now, we consider the iterates $\{f^i\}$ of $f$.  From the condition $f(p)=p$ and the Definition \ref{D-4.1}, $f^i$ is not compactly divergent. Therefore, $f^i$ is relatively compact. The subsequent proof can simply replicate that found in \cite{WHS} or \cite{kobayashi1}.
\end{proof}

In \cite{WHS}, by applying the \textit{basic theorem} of normal families in Hermitian Geometry, Wu proved the Cartan-Carath\'eodory-Kaup-Wu theorem on a Hermitian manifold as follows.

\begin{theorem}\label{gamma} (see \cite{WHS})
Let $M$ be a complex manifold which is either an open relatively compact submanifold of a strongly negatively curved Hermitian manifold $N$, or a complete strongly negatively curved Hermitian manifold itself. Then for a holomorphic mapping $f: M \rightarrow M$ leaves a point $p$ fixed:\\
$(1)$ $|\det df_p|\leq 1$;\\
$(2)$ If $df_p:T_p^{1,0}M \rightarrow T^{1,0}_pM$ is the identity linear map, then $f$ is the identity mapping of $M \rightarrow M$.\\
$(3)$ $|\det df_p|=1$ if and only if $f$ is an automorphism of $N$.
\end{theorem}

Now, we introduce and prove the following theorem.
\begin{theorem}\label{C}
Let $M$  be a complex manifold which is either an open relatively compact submanifold of a strongly negatively curved complex Finsler manifold $(N,G=F^2)$, or a complete strongly negatively curved complex Finsler manifold itself. Then for a holomorphic mapping $f: M \rightarrow M$ leaves a point $p$ fixed:\\
$(1)$ $|\det df_p|\leq 1$;\\
$(2)$ If $df_p:T_p^{1,0}M \rightarrow T^{1,0}_pM$ is the identity linear map, then $f$ is the identity mapping of $M \rightarrow M$.\\
$(3)$ $|\det df_p|=1$ if and only if $f$ is an automorphism of $M$.
\end{theorem}
\begin{proof}
If $M$ is a complex manifold which is either an open relatively compact submanifold of a strongly negatively curved complex Finsler manifold $N$, by  Proposition \ref{P-2.1}, \ref{P-2.3}, we know that $(M,G|_M)$ is a strongly negatively curved complex Finsler manifold. Combining with Corollary \ref{tight}, $N$ is tight.  And by applying Theorem \ref{Wu-3}, the above theorem holds.

If $M$ is a complete strongly negatively curved complex Finsler manifold, by Corollary \ref{tight}, we show that $M$ is taut. And combing with Theorem \ref{Wu-3}, this theorem still holds.
\end{proof}

Now, we suppose that $D_1$ is a bounded domain in $\mathbb{C}^n(n\geq 2)$. We can find a bounded domain $D_2$ in $\mathbb{C}^m(m\geq n)$ such that $D_1 \subsetneq D_2$ and $D_1$ is relatively compact in $D_2$. By Theorem \ref{example}, $D_2$ admits a complex Finsler metric $G$ such that ist holomorphic sectional curvature is bounded from above by a negative constant. Hence $(D_2,G)$ satisfies the conditions of Theorem \ref{C}. We derive the following theorem.
\begin{theorem}(see \cite{WHS})(Cartan-Carath\'eodory-Kaup-Wu Theorem) Let $D$ be a bounded domain in $\mathbb{C}^n$. Then for a holomorphic mapping $f: M \rightarrow M$ leaves a point $p$ fixed:\\
$(1)$ $|\det df_p|\leq 1$;\\
$(2)$ If $df_p:T_p^{1,0}D \rightarrow T^{1,0}_pD$ is the identity linear map, then $f$ is the identity mapping of $D \rightarrow D$.\\
$(3)$ $|\det df_p|=1$ if and only if $f$ is an automorphism of $D$.
\end{theorem}
 However, complex manifolds in the above results are bounded domain in $\mathbb{C}^n$. By Theorem \ref{T-4.7}, we find an generalized example as follows.
 \begin{example}
Suppose that $\pi: \mathcal{X} \rightarrow S$ is an effectively parametrized holomorphic family of compact canonically polarized complex manifolds over a complex manifold $S$. And let $S_1$ be a relatively compact open submanifold of a complex manifold $S$. Then for a holomorphic mapping $f: S_1 \rightarrow S_1$ leaves a point $p$ fixed:\\
$(1)$ $|\det df_p|\leq 1$;\\
$(2)$ If $df_p:T_p^{1,0}S_1 \rightarrow T^{1,0}_pS_1$ is the identity linear map, then $f$ is the identity mapping of $S_1 \rightarrow S_1$.\\
$(3)$ $|\det df_p|=1$ if and only if $f$ is an automorphism of $S_1$.
 \end{example}
 \begin{proof}
By Theorem \ref{T-4.7}, we know that $S$ admits a complex Finsler metric $G$ such that its holomorphic sectional curvature is bounded above by a negative constant. Combining with Proposition \ref{P-2.1}, \ref{P-2.3}, $(S_1, G_{|S_1})$ is a strongly negatively curved complex Finsler manifold. By applying Theorem \ref{C}, the conclusions within the theorem are validated.
 \end{proof}
\section{Applications of Theorem \ref{basic}}
\noindent

In this section, we give some applications of Theorem \ref{basic}. In complex Finsler geometry, the volume element is very complicated. However, there is a Hermitian metric which induces a volume element on a complex manifold. During the proof of Theorems \ref{Wu-1}and \ref{Wu-2}, geometric quantities such as holomorphic curvature were not utilized. Then the Hermitian volume element is still used when we generalize Theorems \ref{Wu-1}and \ref{Wu-2} to a complex Finsler manifold. Theorems \ref{Wu-1}and \ref{Wu-2} are proved by using the basic theorem in Hermitian manifolds, i.e., Theprem \ref{T-1.1}. It is natural that we generalize Theorems \ref{Wu-1}and \ref{Wu-2} to a a complex Finsler manifold by Theorem \ref{basic}.

Firstly, we also need to establish some notations. If $N$ is a Hermitian manifold of complex dimension $n$ with the volume element of $\Omega$.
 Suppose $p_0$ is a fixed point of another complex manifold $M$ also of dimension $n$, and $\mu$ is a fixed nonzero real co-vector of degree $2n$ at $p_0$, i.e. $\mu \in \bigwedge^{2n}T_p^*M$, where $T_p^*M$ denotes the real cotangent space of $M$ at $p_0$ and $\bigwedge$ denotes exterior power. If $f \in \mathcal{H}(M,N)$, then $(f^*\Omega)_{p_0}=c\mu$ for some real number $c$. We shall denote $c$ by  $(f^*\Omega)=c\mu$ for convenience. 

\begin{theorem}\label{A}
Let $M$ be a complex manifold.  And let $N$ be a strongly negatively curved complex Finsler manifold with the volume element $\Omega$ and $\dim M=\dim N=n$. Let $\mathcal{F}_a: M \rightarrow N$ be a family of holomorphic mappings with these properties:\\
$(1)$~~At a fixed point $p_0$ of $M$, $|f^*\Omega/\mu| \geq a >0$ for all $f \in \mathcal{F}_a$ and $\mu$ is a fixed nonzero real vector at $p_0$ of degree $2n$.\\
$(2)$~~$p_0$ gets carried by each $f \in \mathcal{F}_a$ into some fixed compact set $K$ in $N$.\\
Then there is a positive constant $\alpha$ such that each $f \in \mathcal{F}_a$ possesses a univalent ball of radius $\alpha$ around $f(p_0)$.
\end{theorem}
\begin{proof}
Applying the condition that $N$ is a strongly negatively curved complex Finsler manifold to Theorem \ref{basic}, $\mathcal{F}_a$ is equicontinuous. Let $V$ be a relatively compact open neighborhood of the compact $K$ of $(2)$ in $N$ and let distance from $K$ to the complement of $V$ be $\varepsilon$. Choose a neighborhood $U$ of $p_0$ such that $p \in U$ implies $d(p,p_0) <\varepsilon$. The Cauchy integral formula implies in a standard way that $\mathcal{H}(M,N)$ is closed in $\mathcal{C}(M,N)$. And $\mathcal{F}_a$ is closed in $\mathcal{H}(M,N)$. Hence $\mathcal{F}_a$ is closed in $\mathcal{C}(M,N)$. Now if $\mathcal{F}_a$ is considered as a subset of $\mathcal{H}(M,N)$, it is relatively compactly in virtue of Lemma \ref{Ascoli}. Theorem \ref{Wu-1} applies to conclude the proof.
\end{proof}
\begin{theorem}\label{B}
Let $M$ be a Hermitian manifold.  And let $N$ be a strongly negatively curved complex Finsler manifold with the volume element $\Omega$ and $\dim M=\dim N=n$. Let $\mathcal{F}_a: M \rightarrow N$ be a family of holomorphic mappings with these properties:\\
$(1)$~~At a fixed point $p_0$ of $M$, $|f^*\Omega/\mu| \geq a >0$ for all $f \in \mathcal{F}_a$ and $\mu$ is a fixed nonzero real vector at $p_0$ of degree $2n$.\\
$(2)$~~$p_0$ gets carried by each $f \in \mathcal{F}_a$ into some fixed compact set $K$ in $N$.\\
Then there is a positive constant $\lambda$ such that each $f \in \mathcal{F}_a$ is biholomorphic on an open ball of radius $\lambda$ around $p_0$.
\end{theorem}
\begin{proof}
For Theorem \ref{B}, we also consider $\mathcal{F}_a$ as a subset of $\mathcal{H}(U,N)$ and apply Theorem \ref{Wu-1}, where $U$ is given by the proof Theorem \ref{A}.
\end{proof}
\begin{theorem}\label{D}
If a complex Finsler  manifold $N$ is strongly negatively curved, then its automorphism $H(N)$ is a (not necessarily connected) Lie group and the isotropy group of $H(N)$ at a point is a compact Lie group. If $N$ is compact, then $H(N)$ is finite. No complex Lie transformation group of positive dimension acts nontrivially on $N$.
\end{theorem}
\begin{proof}
Combining Corollary \ref{tight} with the condition  that it is a strongly negatively curved complex Finsler manifold, we obtain that $N$ is tight. By Theorem \ref{Wu-4}, the first two results are proved.  Applying the fact that the holomorphic sectional curvature has a negative constant upper bound to Theorem \ref{Wu-7}, we know that $N$ is hyperbolic. The third result is immediate from Theorem \ref{Wu-8}.
\end{proof}
\begin{remark}
By thereom \ref{example}, a bounded domain $D$ in $\mathbb{C}^n(n \geq 2)$ admits a strongly pseudoconvex complex Finsler metric $G_1$ such that $(D, G_1)$ is a strongly negatively curved. Combining with Theorem \ref{D}, we know that the automorphism $H(D)$ of $D$ is a real Lie group  (not a complex Lie group) and the isotropy group of $H(N)$ at a point is a compact Lie group. Suppose that $\pi: \mathcal{X} \rightarrow S$ is an effectively parametrized holomorphic family of compact canonically polarized complex manifolds over a complex manifold $S$. By Theorem \ref{T-4.7}, $S$ admits a $C^{\infty}$ Aut $(\pi)$-invariant Finsler metric$G_2$ whose holomorphic sectional curvature is bounded above by a negative constant. Hence $(S, G_2)$ is strongly negatively curved. Combining with Theorem \ref{D}, the automorphism $H(S)$ of $S$ is a real Lie group and not a complex Lie group.
\end{remark}

By Theorem \ref{SL} (or Theorem \ref{Wu-5} and Corollary \ref{tight}), we obtain the following theorem.
\begin{theorem}\label{E}
Every holomorphic mappings from $\mathbb{C}^n$ into strongly negatively curved complex Finsler manifold $N$ reduces to a constant.
\end{theorem}

These are immediate from Theorems \ref{D}, \ref{E}. The following consequences of Theorem \ref{E} is of particular interest.
\begin{corollary}\label{C-5.2}
$\mathbb{C}^n$ cannot be equipped with a complex Finsler metric whose holomorphic sectional curvature is bounded above by a negative constant $-K_0<0$.
\end{corollary}

Combing Corollary \ref{tight} and Theorem \ref{Wu-6}, we get the following result.
\begin{theorem}\label{F}
If a domain $E$ in $\mathbb{C}^n$ can be given a complete strongly negatively curved complex Finsler metric, then it is a domain of holomorphy.
\end{theorem}

By Theorems \ref{example}, \ref{F} and Corollary \ref{C-5.2}, we have the following corollary.
\begin{corollary}
The bounded domain in $\mathbb{C}^n(n \geq 2)$ is a domian of holomorphy. And $\mathbb{C}^n$ is not a domain of holomorphy.
\end{corollary}

{\bf Acknowledgement:}{ The author thanks Professor Chunping Zhong, Professor Tingbin Cao, Professor Zhenqi Li, Professor Chunhui Qiu, Professor Haiping Fu, Professor Bo Yang, Doctor Hongzhe Cao for their suggestions on the study of complex Finsler manifolds and several complex variables. The author thanks the referees for carefully reading the manuscript and their valuable corrections/suggestions which improved the presentation of the paper. This work was supported by National Natural Science Foundation of China (Grant No. 12071386,  No. 11971401) and the Natural Science Foundation of Jiangxi Provence in China (no. 20232ACB201005).}

\end{document}